\documentclass[12pt,reqno,centertags]{amsart}

\usepackage[utf8]{inputenc}
\usepackage[T1]{fontenc}
\usepackage[american]{babel}
\usepackage[
  babel=true,
  expansion=alltext,
  protrusion=alltext-nott,
  final
]{microtype}
\usepackage[sb]{libertine}
\usepackage[varqu,varl]{inconsolata}
\usepackage[libertine]{newtxmath}
\useosf

\usepackage{amsthm}
\usepackage[margin=3cm,marginparwidth=2.5cm]{geometry}
\usepackage[svgnames]{xcolor}
\usepackage[breaklinks=true,pagebackref]{hyperref}

\hypersetup{
  pdftitle={An improved nonexistence bound for the liquid drop model},
  pdfauthor={Marvin R. Schulz},
  pdfsubject={Mathematics},
  pdfkeywords={}
}

\newcommand{\Sph}{\mathbb{S}}
\newcommand{\R}{\mathbb{R}}
\newcommand{\cB}{\mathcal{B}}
\newcommand{\cE}{\mathcal{E}}
\newcommand{\cF}{\mathcal{F}}
\newcommand{\cH}{\mathcal{H}}
\newcommand{\abs}[1]{\left\lvert #1\right\rvert}

\allowdisplaybreaks

\newtheorem{theorem}{Theorem}[section]
\newtheorem{proposition}[theorem]{Proposition}
\newtheorem{lemma}[theorem]{Lemma}

\theoremstyle{definition}

\theoremstyle{remark}
\newtheorem{remark}[theorem]{Remark}

\numberwithin{equation}{section}

\title{An improved nonexistence bound for the liquid drop model}
\author{Marvin R. Schulz}
\address{Department of Mathematical Sciences, University of Copenhagen, Universitetsparken 5, 2100 Copenhagen, Denmark}
\email{masc@math.ku.dk}
\date{\today}

\begin{document}

\begin{abstract}
	For Gamow's liquid drop model, we improve the nonexistence bound by Frank--Killip--Nam. The key additional input is a geometric perimeter inequality that follows from the capillarity problem of liquids. This yields a quantitative gain in the previously used averaged slicing argument. We prove that the variational problem admits no minimizer for \(A\geq 7.5\).
\end{abstract}
\maketitle

{\hypersetup{linkcolor=black}
	\tableofcontents}
\section{Introduction and Main Theorem}\label{sec:main_theorem}
The liquid drop model of atomic nuclei was introduced and studied in the works of  Gamow \cite{Gamow:1930}, von Weizsäcker \cite{Weizsaecker:1935}, and Bohr \cite{Bohr:1936} in the 1930s.
For a measurable set \(\Omega\subset\R^3\) of finite perimeter \(P(\Omega)\) (in the sense of De Giorgi), consider Gamow's liquid-drop energy
\begin{equation}\label{eq:energy}
	\cE(\Omega)=P(\Omega)+D(\Omega),
	\qquad
	D(\Omega)
	=\frac12\iint_{\Omega\times\Omega}\frac{dx\,dy}{\abs{x-y}}.
\end{equation}
At fixed volume \(A>0\), the corresponding variational problem is
\begin{equation}\label{eq:variational-problem}
	E(A)=\inf\bigl\{\cE(\Omega):
	\Omega\subset\R^3\ \text{measurable},\ \abs{\Omega}=A\bigr\}, \text{ and set } E(0):=0.
\end{equation}
The main theorem is
\begin{theorem}\label{thrm:main}
	If \(E(A)\) is attained, then \(A<7.5\).
\end{theorem}
\begin{remark}
	The proof of this theorem is given in Section~\ref{sec:Proof_main_theorem}.
\end{remark}
Theorem~\ref{thrm:main} improves the bound \(A\leq 8\) found by Frank--Killip--Nam in \cite{FKN:2016}. Note that the perimeter \(P(\Omega)\) is minimized by spheres, whereas the Coulomb interaction favors splitting into distant components.  It is conjectured that \(E(\cdot)\) is minimized by a ball up to the threshold where splitting into two balls is favourable over one ball which is
\begin{equation*}
	A_c = 5\cdot \frac{2-2^{2/3}}{2^{2/3}-1} \approx 3.512.
\end{equation*}
In \cite{CR:2024} by Chodosh and Ruohoniemi, it was recently proven that for \(A\leq 1\), the minimizers are indeed balls. The remainder of this paper is organized as follows. In Section~\ref{sec:FKN_slicing} we recall the proof of Frank--Killip--Nam and identify the step that we improve. In Section~\ref{sec:cap_ineq} we discuss the geometric observations for cutting planes with hypersurfaces following from the capillarity problem and prove in Section~\ref{sec:cut_def_neq} an auxiliary inequality, which is then used in the proof of the main Theorem in Section~\ref{sec:Proof_main_theorem}.
\section{Averaged Slicing Argument}\label{sec:FKN_slicing}
We recall the result of Frank--Killip--Nam.  In the following, \(\cH^2\) denotes the two-dimensional Hausdorff measure, and \(\partial^*\) denotes the reduced boundary. We denote by \(\Omega^{(1)}\) the set of points of density one of
\(\Omega\). Throughout the remainder we assume that \(\Omega\subset \R^3\) with finite perimeter is the minimizer of \(\cE\) with \(\abs{\Omega}=A\) and define for \(\nu\in \Sph^2\) and \(t\in \R\)
\begin{equation*}
	H_{\nu,t} =\{x\in \R^3: x\cdot \nu-t=0\}, \quad H_{\nu,t}^\pm =\{x\in \R^3: \pm(x\cdot \nu-t)>0\}.
\end{equation*}
Let
\begin{equation}\label{eq:symbols_defined}
	\Omega^\pm_{\nu,t} = \Omega \cap H_{\nu,t}^\pm, \quad a^\pm_{\nu,t} = \abs{\Omega^\pm_{\nu,t}}, \quad  \sigma_{\nu,t} = \cH^2(\Omega^{(1)} \cap H_{\nu,t}),
\end{equation}
and define
\begin{equation*}
	s_{\nu,t} = \frac{\min\{a_{\nu,t}^+,a_{\nu,t}^-\}}{A} \in [0,1/2] .
\end{equation*}
We aim to study the following deficit produced by cutting the set in two parts, in particular, we define
\begin{equation}\label{def:deficit}
	g_{\nu,t}
	:=
	\cE(\Omega_{\nu,t}^{+})
	+\cE(\Omega_{\nu,t}^{-})
	-\cE(\Omega),
\end{equation}
and call it the cutting deficit (at \(H_{\nu,t}\)).
\begin{proposition}[Frank--Killip--Nam \cite{FKN:2016}]
	\label{prop:FKN-bound}
	If \(E(A)\) is attained, then \(A\leq 8\).
\end{proposition}
\begin{proof}
	The proof is by contradiction and proceeds as follows: Define
	\begin{equation*}
		I_{\nu,t}
		=
		\iint_{\Omega_{\nu,t}^{+}\times\Omega_{\nu,t}^{-}}
		\frac{dx\,dy}{\abs{x-y}}.
	\end{equation*}
	The cutting formulas for perimeter and Coulomb energy give (for a.e. \(t\in \R\))
	\begin{equation*}
		P(\Omega_{\nu,t}^{+})+P(\Omega_{\nu,t}^{-})
		=
		P(\Omega)+2\sigma_{\nu,t},
	\end{equation*}
	and
	\begin{equation*}
		D(\Omega_{\nu,t}^{+})+D(\Omega_{\nu,t}^{-})
		=
		D(\Omega)-I_{\nu,t}.
	\end{equation*}
	Consequently, for the cutting deficit
	\begin{equation}\label{eq:deficit}
		g_{\nu,t}
		=
		\cE(\Omega_{\nu,t}^{+})
		+\cE(\Omega_{\nu,t}^{-})
		-\cE(\Omega)
		=
		2\sigma_{\nu,t}-I_{\nu,t}.
	\end{equation}
	By minimality of \(\Omega\) and subadditivity of \(E\) it is
	\begin{equation*}
		\begin{aligned}
			\cE(\Omega_{\nu,t}^{+})
			+\cE(\Omega_{\nu,t}^{-})
			 & \geq E(a_{\nu,t}^+)+E(A-a_{\nu,t}^+) \\
			 & \geq E(A)
			=\cE(\Omega).
		\end{aligned}
	\end{equation*}
	Consequently, \(g_{\nu,t}\geq 0\). First integrating over \(t\in \R\) it follows by Cavalieri's principle
	\begin{equation*}
		\int_\R g_{\nu,t} dt = 2A - \int_{\R}I_{\nu,t}\,dt
		=2A-
		\iint_{\Omega\times\Omega}
		\frac{(\nu\cdot(x-y))_{+}}{\abs{x-y}}\,dx\,dy.
	\end{equation*}
	Taking now the average over \(\nu \in \Sph^2\) yields
	\begin{equation}\label{eq:main_FKN}
		\fint_{\Sph^2} \int_\R g_{\nu,t} dt d\nu = 2A - \frac{A^2}{4} = \frac{A(8-A)}{4} .
	\end{equation}
	Since the left-hand side above is nonnegative, the result follows.
\end{proof}
\begin{remark}
	We aim to prove a positive lower bound on the cutting deficit \(g_{\nu,t}\) which is comparable to \(\sigma_{\nu,t}\). After interchanging the two half-spaces if necessary, we assume that
\[
    |\Omega_{\nu,t}^+|=s_{\nu,t}A,
    \qquad
    |\Omega_{\nu,t}^-|=(1-s_{\nu,t})A.
\]
	The cutting deficit \(g_{\nu,t}\) then does decompose into three terms
	\begin{equation}\label{eq:g_def_dec}
		\begin{split}
			g_{\nu,t} = & [\cE(\Omega_{\nu,t}^+)- E(s_{\nu,t}A)]      \\
			+           & [\cE(\Omega_{\nu,t}^-)- E((1-s_{\nu,t})A)]  \\
			+           & [E(s_{\nu,t}A) + E((1-s_{\nu,t})A) - E(A)].
		\end{split}
	\end{equation}
	We will show that they cannot always vanish at the same time using the capillarity isoperimetric inequalities discussed in the next section.
\end{remark}
\section{Capillarity Inequalities}\label{sec:cap_ineq}
A drop of liquid resting on a flat surface \(H_{\nu,t}\) minimizes the competition between its free surface and the contact area with the plane. This behavior is described by the following capillarity functional
\begin{equation*}
	\cF_\lambda(F)
	=
	P(F;H_{\nu,t}^+)
	-
	\lambda
	\cH^2(\partial^*F\cap H_{\nu,t}),
	\qquad
	\lambda\in(-1,1),
\end{equation*}
defined for \(F\subset H_{\nu,t}^+\) with finite perimeter. Here \(P(F;H_{\nu,t}^+)\) denotes the relative perimeter of \(F\) inside the
half-space \(H_{\nu,t}^+\) and
\begin{equation*}
	\cH^2(\partial^*F \cap H_{\nu,t})
\end{equation*}
is the contact area of \(F\) with \(H_{\nu,t}\). The parameter \(\lambda\) measures the preference of
liquid adhered to the supporting plane. The limiting cases
\(\lambda=-1\) and \(\lambda=1\) correspond to complete non-wetting and complete
wetting, respectively, while \(\lambda=0\) represents the neutral case, where the liquid touches the surface at \(90\)--degree angle.

The minimizers of the capillarity problem under fixed volume are spherical caps, and the variational characterization yields a sharp inequality for every finite-perimeter set. More precisely, there exists the following isoperimetric inequality
\begin{lemma}\label{lem:cap_iso}
	Let \(F\subset H_{\nu,t}^+\) be measurable with finite perimeter and \(\abs{F}<\infty\), then
	\begin{equation*}
		\cF_\lambda(F) \geq  \left(\frac{2-3\lambda+\lambda^3}{4}\right)^{1/3}\beta(\abs{F}).
	\end{equation*}
	where \(\beta(V):=(36\pi)^{1/3}V^{2/3}\) denotes the perimeter of a ball of volume \(V\).
\end{lemma}
\begin{proof}
	After translation we may move \(H_{\nu,t}\) onto the half space \(\{x_3>0\}\) and  using \cite[Theorem~3.3]{PP:2024} we find
	\begin{equation*}
		\cF_\lambda(F) \geq 3 \abs{B^\lambda}^{1/3}\abs{F}^{2/3},
	\end{equation*}
	where
	\( B^\lambda
	:=
	B_1(0)\cap\{x\in\R^3:x_3>\lambda\} \).
	By a direct computation
	\begin{equation*}
		\abs{B^\lambda}=
		\frac{\pi}{3}
		\left(2-3\lambda+\lambda^3\right).
	\end{equation*}
	Consequently,
	\begin{equation*}
		\begin{split}
			3\abs{B^\lambda}^{1/3}\abs{F}^{2/3}
			 & =
			3\left(\frac{\pi}{3}
			\left(2-3\lambda+\lambda^3\right)\right)^{1/3}\abs{F}^{2/3} \\
			 & =
			\left(\frac{2-3\lambda+\lambda^3}{4}\right)^{1/3}
			(36\pi)^{1/3}\abs{F}^{2/3}                                  \\
			 & =
			\left(\frac{2-3\lambda+\lambda^3}{4}\right)^{1/3}
			\beta(\abs{F}).
		\end{split}
	\end{equation*}
\end{proof}
For our application, it is more convenient to express this inequality in terms of the perimeter \(P(F)\). We show
\begin{lemma}\label{lem:isoper_lambda}
	Let \(F\subset H_{\nu,t}^+\) be measurable with finite perimeter and \(\abs{F}<\infty\), then for any \(\lambda\in[-1,1]\)
	\begin{equation}\label{eq:capillarity}
		P(F)
		\geq
		(1+\lambda)\cH^2(\partial^*F\cap H_{\nu,t})
		+
		\left(\frac{2-3\lambda+\lambda^3}{4}\right)^{1/3}
		\beta(\abs{F}).
	\end{equation}
\end{lemma}
\begin{proof}
	For \(\lambda \in (-1,1)\) this follows directly from the identity
	\begin{equation*}
		P(F)
		=
		P(F;H_{\nu,t}^+)
		+\cH^2(\partial^*F\cap H_{\nu,t}),
	\end{equation*}
	Applying Lemma~\ref{lem:cap_iso} yields the result for \(\lambda\in(-1,1)\) and the endpoint cases \(\lambda=\pm1\) follow by taking a limit.
\end{proof}
\begin{remark}
	We relate the contact area in Lemma~\ref{lem:isoper_lambda} to the quantity \(\sigma_{\nu,t}\) defined in \eqref{eq:symbols_defined} as
	\begin{equation*}
		\sigma_{\nu,t}=\cH^2(\Omega^{(1)} \cap H_{\nu,t}).
	\end{equation*}
	First note that for every fixed \(\nu\in\Sph^2\) and almost every \(t\in \R\) we have
	\begin{equation*}
		\cH^2(\partial^*\Omega\cap H_{\nu,t})=0.
	\end{equation*}
	Applying the formulas for intersections of sets of finite perimeter from \cite[Theorem~16.3]{Maggi} to \(\Omega\cap H_{\nu,t}^{\pm}\), we therefore obtain, up to an \(\cH^2\)-negligible set,
	\begin{equation*}
		\partial^*\Omega_{\nu,t}^{\pm}\cap H_{\nu,t} = \Omega^{(1)}\cap H_{\nu,t}, \quad \text{ a.e. } t\in \R.
	\end{equation*}
	Consequently,
	\begin{equation}\label{eq:contact_area_slice}
		\cH^2\bigl(\partial^*\Omega_{\nu,t}^{\pm}\cap H_{\nu,t}\bigr)
		= \cH^2\bigl(\Omega^{(1)}\cap H_{\nu,t}\bigr)
		= \sigma_{\nu,t}.
	\end{equation}
	Lemma~\ref{lem:isoper_lambda}, applied to \(F=\Omega_{\nu,t}^{+}\), thus gives
	\begin{equation}\label{eq:capillarity_cut}
		P(\Omega_{\nu,t}^{+})
		\geq
		(1+\lambda)\sigma_{\nu,t}
		+
		\left(\frac{2-3\lambda+\lambda^3}{4}\right)^{1/3}
		\beta(a_{\nu,t}^{+})
	\end{equation}
for every \(\lambda\in[-1,1]\) and almost every \(t\in\R\). The analogous estimate holds for \(\Omega_{\nu,t}^{-}\). Throughout the following, all identities and estimates for the slices \(\Omega_{\nu,t}^{\pm}\) that rely on these relations are understood to hold for almost every \(t\in \R\).
\end{remark}
We use the following relation between Coulomb energy, volume and perimeter
\begin{lemma}\label{lem:coulomb_est}
	If \(F\subset \R^3\) is measurable with \(\abs{F}<\infty\) and finite perimeter, then
	\begin{equation*}
		3 P(F)^2 D(F) \geq 16\pi \abs{F}^3.
	\end{equation*}
\end{lemma}
\begin{proof}
	Following \cite[Lemma 18 and Corollary 19]{CR:2024} one has
	\begin{equation*}
		\abs{F}^3 \leq \frac{27}{4e_\ast^3} P(F)^2 D(F),
	\end{equation*}
	where
	\begin{equation*}
		e_\ast := \min_{F \subset \R^3} \frac{\cE(F)}{\abs{F}} \geq (36\pi)^{1/3}
	\end{equation*}
	with the minimum taken over all measurable sets with $0<\abs{F}<\infty$.
\end{proof}
\section{Control of Cutting Deficit}\label{sec:cut_def_neq}
Fix \(\nu \in \Sph^2\) and \(t\in \R\), we aim to prove a lower bound on the cutting deficit \(g_{\nu,t}\). Since \(g_{\nu,t}\), \(s_{\nu,t}\), and \(\sigma_{\nu,t}\) are
unchanged when the two half-spaces are interchanged, we label them so
that
\begin{equation*}
	\abs{\Omega_{\nu,t}^+}=s_{\nu,t}A,
	\qquad
	\abs{\Omega_{\nu,t}^-}=(1-s_{\nu,t})A.
\end{equation*}
Note that we can decompose \(g_{\nu,t}\) as follows
\begin{equation}\label{eq:decomposed}
	\begin{split}
		g_{\nu,t} = & [\cE(\Omega_{\nu,t}^+)- E(s_{\nu,t}A)]      \\
		+           & [\cE(\Omega_{\nu,t}^-)- E((1-s_{\nu,t})A)]  \\
		+           & [E(s_{\nu,t}A) + E((1-s_{\nu,t})A) - E(A)].
	\end{split}
\end{equation}
Each of the three contributions on the right-hand side of \eqref{eq:decomposed} is nonnegative. The first two by the definition of \(E(\cdot)\) and the third one due to subadditivity.
For fixed \(\nu\in \Sph^2\) and \(t\in \R\) we will now write
\begin{equation*}
	g_{\nu,t} =g,\quad s_{\nu,t} =s, \quad \sigma_{\nu,t} = \sigma,
\end{equation*}
and assume \(sA\leq 1\). We will estimate the terms independently by proving two lemmas
\begin{lemma}\label{lem:Q_lemma}
	Let \(\nu\in \Sph^2\) and \(t\in \R\) fixed with  \(0<sA\leq 1\) and \(\sigma\neq 0\), define for \(x\geq 0\)
	\begin{equation*}
		Q(x) =  \sup_{\lambda \in [-1,1]} Q_\lambda(x), \quad Q_\lambda(x) =  (1+\lambda)x +  \left(\frac{2-3\lambda+\lambda^3}{4}\right)^{1/3}
	\end{equation*}
	then
	\begin{equation}\label{statement2}
		\cE(\Omega_{\nu,t}^+) - E(sA) \geq \beta(sA) \left[Q(x)  +\frac{4sA}{27 Q(x)^2} - \left(1+\frac{sA}{5}\right)\right]_+, \quad x= \frac{\sigma}{\beta(sA)}>0.
	\end{equation}
\end{lemma}
\begin{remark}
	Optimizing in \(\lambda\in [-1,1]\) indeed gives the explicit representation
	\begin{equation*}
		Q(x)
		=
		\sqrt[3]{x^3+\frac{1}{2}+\sqrt{x^3+\frac{1}{4}}}
		+
		\sqrt[3]{x^3+\frac{1}{2}-\sqrt{x^3+\frac{1}{4}}}.
	\end{equation*}
	In the proof of Lemma~\ref{lem:Lip_L} we show the estimate
	\begin{equation*}
		1\leq Q(x) \leq 1+\frac{4}{3}x^2.
	\end{equation*}
\end{remark}
\begin{proof}[Proof of Lemma~\ref{lem:Q_lemma}]
	Applying Lemma~\ref{lem:isoper_lambda} we find
	\begin{equation}\label{eq:per_q_bnd}
		P(\Omega_{\nu,t}^+) \geq (1+\lambda)\sigma
		+
		\left(\frac{2-3\lambda+\lambda^3}{4}\right)^{1/3}
		\beta(sA) = \beta(sA) Q_\lambda(x).
	\end{equation}
	For the Coulomb term, we note that due to Lemma~\ref{lem:coulomb_est}
	\begin{equation*}
		D(\Omega_{\nu,t}^+) \geq \frac{16\pi (sA)^3}{3 P(\Omega_{\nu,t}^+)^2}.
	\end{equation*}
	Thus we have
	\begin{equation*}
		\cE(\Omega_{\nu,t}^+) \geq P(\Omega_{\nu,t}^+) +  \frac{16\pi (sA)^3}{3 P(\Omega_{\nu,t}^+)^2}.
	\end{equation*}
	We use that  for \(p\geq \beta(sA)\) and \(sA\leq 1\)
	the function
	\begin{equation*}
		f(p)=p+ \frac{16\pi (sA)^3}{3 p^2}
	\end{equation*}
	is strictly increasing since
	\begin{equation*}
		f'(p)\geq \frac{19}{27}>0.
	\end{equation*}
	Consequently
	\begin{equation*}
		\cE(\Omega_{\nu,t}^+)\geq f(P(\Omega_{\nu,t}^+)) \geq f(\beta(sA)Q(x)),
	\end{equation*}
	or equivalently
	\begin{equation}\label{eq:cE_low_Q}
		\cE(\Omega_{\nu,t}^+)\geq \beta(sA) Q(x) + \frac{16\pi (sA)^3}{3 \beta(sA)^2 Q(x)^2} = \beta(sA)\left(Q(x)  +\frac{4sA}{27 Q(x)^2}\right).
	\end{equation}
	Since \(sA\leq 1\) it follows from \cite{CR:2024} that functional \(\cE\) is minimized by the ball \(B(0,sA)\) of volume \(sA\) and thus by direct computations
	\begin{equation}\label{eq:CR}
		E(sA)
		=
		\cE(B(0,sA))
		=
		\beta(sA)\left(1+\frac{sA}{5}\right),
	\end{equation}
	and noting \(\cE(\Omega_{\nu,t}^+) \geq E(sA)\)  the claimed inequality in \eqref{statement2} follows immediately.
\end{proof}
Next, we aim to estimate the third term on the right-hand side of \eqref{eq:g_def_dec}. For this we state and prove the following Lemma
\begin{lemma}\label{lem:k_lemma}
	Let \(A\in[7,8]\), \(0<sA\leq 1\) and define
	\begin{equation}\label{eq:def_c0}
		c_0(s) = (3s)^{-2/3}[1-(1-s)^{5/3}]
	\end{equation}
	and
	\begin{equation*}
		\begin{split}
			k(s;A):=
			1+s^{1/3}(1-s)^{2/3}
			-3 c_0(s) + \frac{A}{5}(s-c_0(s)).
		\end{split}
	\end{equation*}
	Then
	\begin{equation*}
		\cB(s;A) := E(sA) + E((1-s)A) - E(A) \geq \beta(sA)[k(s;A)]_+.
	\end{equation*}
\end{lemma}
\begin{proof}
	By \cite[Lemma~5]{FrankNam2021}, for every
	\(\theta\in(0,1)\),
	\begin{equation}\label{eq:Frank-Nam}
		E(\theta A)
		\geq
		\theta^{5/3}E(A)
		+(1-\theta)\theta^{2/3}\beta(A).
	\end{equation}
	Using \(sA\leq1\), the result of
	Chodosh--Ruohoniemi \cite{CR:2024} gives
	\begin{equation*}
		E(sA)
		=
		\beta(sA)\left(1+\frac{sA}{5}\right).
	\end{equation*}
	Applying \eqref{eq:Frank-Nam} with \(\theta=1-s\) and
	using \eqref{eq:CR}, we obtain
	\begin{equation}\label{eq:calB_low_bnd}
		\begin{aligned}
			\cB(s;A)
			 & \geq
			\beta(sA)\left(1+\frac{sA}{5}\right)
			+(1-s)^{5/3}E(A)
			+s(1-s)^{2/3}\beta(A)-E(A)
			\\
			 & =
			\left[1+\frac{sA}{5}
			+s^{1/3}(1-s)^{2/3}\right]\beta(sA)
			-\left[1-(1-s)^{5/3}\right]E(A),
		\end{aligned}
	\end{equation}
	where \(\beta(sA)=s^{2/3}\beta(A)\). Next, we bound \(E(A)\) from above. The strategy is to consider a trial configuration of \(n\) balls such that the volume of the \(n\) balls adds up to \(A\). Following \cite[Section 4.2]{FL:2015} and in particular \cite[Proposition 4.3]{FL:2015} among cuts into equal sized balls \(n=3\) is in the considered range \(A\in[7,8]\) and thus
	by Newton's theorem
	\begin{equation}\label{eq:Newton_applied}
		D(B(0,A/3)) = \frac{A}{15} \beta(A/3), \quad \text{ and } \quad \cE(B(0,A/3)) = \beta(A/3)\left(1+\frac{A}{15}\right).
	\end{equation}
	Using \eqref{eq:Newton_applied} we find
	\begin{equation}\label{eq:up_bnd_2}
		E(A) \leq 3^{-2/3}\left(3+\frac{A}{5}\right) \beta(A).
	\end{equation}
	Inserting \eqref{eq:up_bnd_2} into \eqref{eq:calB_low_bnd} yields
	\begin{equation*}
		\cB(s;A)\geq \beta(sA)\left(\left[1+\frac{sA}{5} +s^{1/3}(1-s)^{2/3}\right]-\left[1-(1-s)^{5/3}\right]3^{-2/3}\left(3+\frac{A}{5}\right)s^{-2/3}\right).
	\end{equation*}
	On the other hand, subadditivity gives
\(\mathcal B(s;A)\geq0\). Combining the two estimates yields
\[
    \mathcal B(s;A)
    \geq \beta(sA)[k(s;A)]_+.
\]
\end{proof}
We combine the previous two lemmas to prove
\begin{lemma}\label{lem:Main}
    Let \(A\in[7,8]\) and fix \(\nu\in\Sph^2\). For \(0<sA\leq1\), define
    \begin{equation*}
        L(s;A)
        :=
        \inf_{x>0}
        \frac{1}{x}
        \left\{
            [k(s;A)]_+
            +
            \left[
                Q(x)
                +
                \frac{4sA}{27Q(x)^2}
                -
                \left(1+\frac{sA}{5}\right)
            \right]_+
        \right\},
    \end{equation*}
    and set \(L(0;A):=0\). Then, for almost every \(t\in\R\) such that
    \(s_{\nu,t}A\leq1\),
    \begin{equation}\label{eq:cutting_deficit_lower_bound}
        g_{\nu,t}
        \geq
        \sigma_{\nu,t}L(s_{\nu,t};A).
    \end{equation}
\end{lemma}
\begin{proof}
	If \(\sigma=0\), the conclusion follows directly from \(g\geq 0\). Thus, we only consider \(\sigma>0\).  For fixed \(\nu\in \Sph^2\) it follows from Cavalieri's principle
	\begin{equation*}
		a_{\nu,t}^+ = \int_t^\infty \sigma_{\nu,z} dz.
	\end{equation*}
	Hence \(t\mapsto a_{\nu,t}^+\) is absolutely continuous and
	\begin{equation*}
		\partial_t a_{\nu,t}^+
		=-\sigma_{\nu,t}
		\qquad\text{for almost every }t\in\R.
	\end{equation*}
	An absolutely continuous function has derivative zero almost everywhere
	on each of its level sets. Consequently,
	\begin{equation*}
		\sigma_{\nu,t}=0
		\qquad\text{for almost every }t
		\text{ such that }a_{\nu,t}^+\in\{0,A\}.
	\end{equation*}
	Since \(s_{\nu,t}=0\) precisely when
	\(a_{\nu,t}^+\in\{0,A\}\), we conclude that
	\begin{equation*}
		g_{\nu,t}\geq0
		=\sigma_{\nu,t}L(0;A)
	\end{equation*}
	for almost every \(t\) such that \(s_{\nu,t}=0\). Thus it remains to
	consider
	\begin{equation*}
		s=s_{\nu,t} \, \text{ such that } 0<sA\leq 1
	\end{equation*}
	Therefore we have \(\sigma >0\) and \(\beta(sA)>0\) and define the dimensionless quantity
	\begin{equation*}
		x=\frac{\sigma}{\beta(sA)}>0.
	\end{equation*}
	Dropping the second nonnegative term in the decomposition
	\eqref{eq:decomposed}, we obtain
	\begin{equation*}
		g
		\geq
		\bigl[\cE(\Omega_{\nu,t}^+)-E(sA)\bigr]
		+
		\cB(s;A).
	\end{equation*}
	Applying the two preceding lemmas therefore gives
	\begin{equation*}
		g\geq
		\beta(sA)
		\left\{[k(s;A)]_+ + \left[Q(x)
			+
			\frac{4sA}{27Q(x)^2}
			-
			\left(1+\frac{sA}{5}\right)\right]_+\right\}.
	\end{equation*}
	Since \(\beta(sA)=\sigma/x\) the statement follows immediately.
\end{proof}
\section{Proof of the Main Theorem}\label{sec:Proof_main_theorem}
In this section, we prove Theorem~\ref{thrm:main}. This argument follows from a bootstrap argument. We state two Lemmas about the function \(L\) defined in Lemma~\ref{lem:Main} and give the proof of Theorem~\ref{thrm:main} under the assumption that these Lemmas hold and defer their Proof to Section \ref{subsec:1} and Section~\ref {subsec:2}. We always assume \(A\in[7.5,8]\) otherwise, nothing is to prove.
\begin{lemma}\label{lem:Lip_L}
	For any fixed \(s \in (0,1/8] \) the function \(L(s;\cdot)\) is Lipschitz, meaning for \(A_1,A_2\in [7,8]\) with \(A_2>A_1\)
	\begin{equation*}
		0\leq L(s;A_1)-L(s;A_2) \leq \frac{\sqrt{5}}{7}s^{-1/6}(A_2-A_1).
	\end{equation*}
\end{lemma}
\begin{proof}
	The proof of Lemma~\ref{lem:Lip_L} is provided in Subsection~\ref{subsec:1}.
\end{proof}
\begin{lemma}\label{lem:actual_int}
	Define
	\begin{equation*}
		I(A) := \int_0^{1/8}L(s;A) ds \quad \text{ then } \quad I(7.5)>\frac{1}{16}.
	\end{equation*}
\end{lemma}
\begin{proof}
	The Proof of Lemma~\ref{lem:actual_int} is provided in Subsection~\ref{subsec:2}.
\end{proof}
Assuming Lemmas \ref{lem:Lip_L} and \ref{lem:actual_int} for now we prove
\begin{theorem}[Copy of Theorem~\ref{thrm:main}]
	If \(E(A)\) is attained, then \(A<7.5\).
\end{theorem}
\begin{proof}
	If $A<7.5$ there is nothing to show, so we assume \(A\in[7.5,8]\). Fix \(\nu\in\Sph^2\), retain the original orientation of the half-spaces, and write \begin{equation*}
		a_\nu(t):=a^+_{\nu,t} =\abs{\Omega\cap H_{\nu,t}^+} =\abs{\Omega^+_{\nu,t}},
	\end{equation*}
	to emphasize that this is a function of \(t\in\R\). We then have, due to Cavalieri's principle
	\begin{equation*}
		a_{\nu}(t) = \int_{t}^\infty \sigma_{\nu,\tau} d\tau,\quad  a_\nu'(t) = - \sigma_{\nu,t}, \quad \text{ for a.e. } t\in \R.
	\end{equation*}
	Moreover
	\begin{equation*}
		\lim_{t\to -\infty} a_{\nu,t}^+ = A, \quad \lim_{t\to +\infty} a_{\nu,t}^+ = 0,
	\end{equation*}
	and the condition \(0\leq s_{\nu,t}A \leq 1\) is equivalent to
	\begin{equation*}
		a_{\nu}(t) \in \left[0, 1\right] \cup [A-1,A].
	\end{equation*}
	Therefore by introducing \(u(t)=a_\nu(t)\) with \(du=-\sigma_{\nu,t} dt\) we find for every fixed \(\nu \in \Sph^2\)
	\begin{equation}\label{eq:cav_part}
		\begin{split}
			\int_{\{t:s_{\nu,t}A\leq 1\}} \sigma_{\nu,t} L\left(s_{\nu,t};A\right) dt & = \int_0^{1} L\left(\frac{u}{A};A\right) du + \int_{A-1}^A L\left(1-\frac{u}{A};A\right) du \\
			                                                                          & = 2A \int_{0}^{\frac{1}{A}} L(s;A) ds
		\end{split}
	\end{equation}
	Combining \eqref{eq:cav_part} with \eqref{eq:main_FKN} by application of Lemma~\ref{lem:Main} we find
	\begin{equation*}
		2A \int_{0}^{\frac{1}{A}} L(s;A) ds \leq \fint_{\Sph^2} \int_\R g_{\nu,t} dt d\nu = \frac{A(8-A)}{4},
	\end{equation*}
	or equivalently
	\begin{equation}\label{eq:start_bootstrap}
		A\leq 8-8 \int_{0}^{\frac{1}{A}} L(s;A) ds.
	\end{equation}
    Let
    \begin{equation}
        I(A) := \int_{0}^{1/8} L(s;A) ds, \quad \Phi(A) := 8-8I(A).
    \end{equation}
	Since \(A\leq 8\) by assumption and \(L\geq 0\) we then conclude from \eqref{eq:start_bootstrap}
    \begin{equation}
        A \leq  8-8 \int_{0}^{\frac{1}{A}} L(s;A) ds \leq 8-8I(A) = \Phi(A).
    \end{equation}
	By Lemma~\ref{lem:Lip_L}, for \(15/2\leq A_1\leq A_2\leq8\),
	\begin{equation*}
		0
		\leq
		L(s;A_1)-L(s;A_2)
		\leq
		\frac{\sqrt5}{7}s^{-1/6}(A_2-A_1).
	\end{equation*}
	Integrating over \((0,1/8)\) gives
	\begin{equation*}
		0
		\leq
		I(A_1)-I(A_2)
		\leq
		\frac{3\sqrt{10}}{140}(A_2-A_1).
	\end{equation*}
	Consequently,
	\begin{equation}\label{eq:Phi-contraction}
		0
		\leq
		\Phi(A_2)-\Phi(A_1)
		\leq 8(I(A_1)-I(A_2))\leq
		\frac{6\sqrt{10}}{35}(A_2-A_1), \quad \text{ with } \frac{6\sqrt{10}}{35}<1.
	\end{equation}
	Thus, \(A\mapsto \Phi(A)-A\) is strictly decreasing on \([7.5,8]\). Therefore, for any \(B\in [7.5,8]\) with \(\Phi(B)<B\) we conclude \(A<B\). To prove \(A<B=7.5\) it therefore suffices to show
	\begin{equation}\label{eq:last_step}
		I(7.5) = \int_0^{1/8}L(s;7.5) ds > \frac{1}{16},
	\end{equation}
	since then
	\begin{equation*}
		\Phi(7.5)= 8-8I(7.5) < 8-\frac{1}{2}=7.5.
	\end{equation*}
	Consequently, the assertion of the main theorem follows from Lemmas~\ref{lem:Lip_L} and~\ref{lem:actual_int}.
\end{proof}

\subsection{Proof of Lemma~\ref{lem:Lip_L}}\label{subsec:1}
\begin{proof}
	For \(Q\) defined in Lemma~\ref{lem:Q_lemma} we set
	\begin{equation*}
		H(s,A,x):= Q(x) + \frac{4sA}{27 Q(x)^2} - \left(1+\frac{sA}{5}\right),
	\end{equation*}
	and with \(k\) defined in Lemma~\ref{lem:k_lemma}
	\begin{equation*}
		\ell(s,A,x):=\frac{[k(s;A)]_++[H(s,A,x)]_+}{x},
	\end{equation*}
	such that by definition
	\begin{equation*}
		L(s;A) = \inf_{x>0} \ell(s,A,x).
	\end{equation*}
	We first show that for \(A\geq 7\)
	\begin{equation}\label{eq:x_0_bnd}
		L(s;A) = \inf_{x\geq x_\ast} \ell(s,A,x), \quad \text{ for } \quad  x_\ast(s)= 7 \left(\frac{s}{180}\right)^{1/2}.
	\end{equation}
	For this, we start by showing
	\begin{equation}\label{eq:Q_est}
		1\leq Q(x) \leq 1+\frac{4}{3}x^2
	\end{equation}
	The lower bound follows by choosing \(\lambda=-1\) in the Definition of \(Q\). For the upper bound, note that
	\begin{equation*}
		Q(x) = \sup_{u\in[0,2]}\left(ux + \left(\frac{(2-u)^2(1+u)}{4}\right)^{1/3}\right)
	\end{equation*}
	and, by cubing and expanding, for \(u\in[0,2]\) one has
	\begin{equation*}
		\left(\frac{(2-u)^2(1+u)}{4}\right)^{1/3}\leq 1-\frac{3u^2}{16}.
	\end{equation*}
	Therefore
	\begin{equation*}
		Q(x) \leq \sup_{u\in[0,2]}\left(1+ux - \frac{3u^2}{16}\right)
	\end{equation*}
	and using \(2ab\leq a^2+b^2\) yields
	\begin{equation*}
		ux\leq \frac{3}{16}u^2  + \frac{4}{3}x^2.
	\end{equation*}
	Therefore, we have the inequality in \eqref{eq:Q_est}. As a consequence
	\begin{equation*}
		H(s,A,x) < \frac{4}{3}x^2 - \frac{7sA}{135}
	\end{equation*}
	and therefore \(H(s,A,x)<0\) for \(x\leq x_\ast\) such that
    \begin{equation}
        \ell(s,A,x)= \frac{[k(s;A)]_+}{x}\geq \frac{[k(s;A)]_+}{x_\ast}, \quad x\leq x_\ast.
    \end{equation}
    Since \(H(s,A,x_\ast)<0\), the right-hand side equals \(\ell(s,A,x_\ast)\). It follows that for $s\neq 0$
\begin{equation*}
    L(s;A)
    =
    \inf_{x\geq x_\ast}\ell(s,A,x).
\end{equation*}
	We continue by showing that \(k(s,A)\) decreases in \(A\). Indeed
	\begin{equation*}
		\partial_A k(s,A) = \frac{s-c_0(s)}{5}
	\end{equation*}
	and for \(sA\leq 1\) and consequently \(s<1/3\) we have \(c_0(s)>s\). Using that \([\cdot]_+\) is \(1\)-Lipschitz we then have
	\begin{equation}\label{eq:k_lip}
		\begin{split}
			0 & \leq [k(s;A_1)]_+ - [k(s;A_2)]_+ \\
			  & \leq k(s;A_1) - k(s;A_2)         \\
			  & =  \frac{A_2-A_1}{5}(c_0(s)-s).
		\end{split}
	\end{equation}
	Also \(H(s,A,x)\) is decreasing in \(A\) since
	\begin{equation*}
		H(s,A,x) = Q(x)-1 + sA\left(\frac{4}{27Q(x)^2}-\frac{1}{5}\right)
	\end{equation*}
	and with \(Q(x)\geq 1\)
	\begin{equation*}
		\frac{4}{27Q(x)^2}-\frac{1}{5} <0.
	\end{equation*}
	Using again that  \([\cdot]_+\) is \(1\)-Lipschitz we find
	\begin{equation}\label{eq:H_lip}
		\begin{split}
			0 & \leq [H(s,A_1,x)]_+ - [H(s,A_2,x)]_+                      \\
			  & \leq H(s;A_1,x) - H(s;A_2,x)                              \\
			  & = s(A_2-A_1)\left(\frac{1}{5} - \frac{4}{27Q(x)^2}\right).
		\end{split}
	\end{equation}
	Using \eqref{eq:k_lip} and \eqref{eq:H_lip} we find
	\begin{equation*}
		0\leq  \ell(s,A_1,x)- \ell(s,A_2,x) \leq \frac{A_2-A_1}{x}\left(\frac{c_0(s)}{5} - \frac{4s}{27 Q(x)^2}\right) \leq \frac{c_0(s)}{5x}(A_2-A_1)
	\end{equation*}
    and consequently
    \begin{equation}
        \inf_{x\geq x_\ast}  \ell(s,A_1,x) - \inf_{x\geq x_\ast} \ell(s,A_2,x) \leq \sup_{x\geq x_\ast}\left( \ell(s,A_1,x)- \ell(s,A_2,x)\right) \leq \frac{c_0(s)}{5x_\ast}(A_2-A_1). 
    \end{equation}
	By definition of $L$, we then arrive at
	\begin{equation*}
		L(s;A_1)-L(s;A_2) \leq \frac{c_0(s)}{5 x_\ast} (A_2-A_1).
	\end{equation*}
	Using the definition of \(x_\ast\) in \eqref{eq:x_0_bnd} and the estimate on \(c_0\) defined in \eqref{eq:def_c0}
	\begin{equation*}
		c_0(s) =  (3s)^{-2/3}[1-(1-s)^{5/3}] \leq \frac{5}{3^{5/3}} s^{1/3}
	\end{equation*}
	we arrive at
	\begin{equation*}
		L(s;A_1)-L(s;A_2) \leq \frac{\sqrt{180}}{3^{5/3} \cdot 7}  s^{-1/6}   (A_2-A_1) < \frac{\sqrt{5}}{7} s^{-1/6} (A_2-A_1).
	\end{equation*}
\end{proof}
\subsection{Proof of Lemma~\ref{lem:actual_int}}\label{subsec:2}
Before integrating, we derive a suitable lower bound on \(L(s;A)\). For \(H\) defined in the preceding proof we have
\begin{equation*}
	L(s;A) = \inf_{x>0} \frac{[k(s;A)]_++[H(s,A,x)]_+}{x},
\end{equation*}
and
\begin{equation*}
	H(s,A,x) =Q(x) + \frac{4sA}{27 Q(x)^2} - \left(1+\frac{sA}{5}\right).
\end{equation*}
We use the estimate (for \(sA\leq 1\))
\begin{equation*}
	q + \frac{4sA}{27 q^2} \geq \alpha q + \frac{4sA}{9}, \quad \alpha := 1-\frac{8sA}{27} >0
\end{equation*}
to find
\begin{equation*}
	H(s,A,x)\geq \alpha Q(x) + \frac{4sA}{9} - (1+\frac{sA}{5}).
\end{equation*}
By definition of \(Q_\lambda\) we have for any \(\lambda\in [-1,1]\) then
\begin{equation}\label{eq:lambda_back}
	H(s,A,x)\geq \alpha Q_\lambda(x) +\frac{11}{45}sA-1.
\end{equation}
with
\begin{equation}\label{eq:q_lambda}
	Q_\lambda(x)= (1+\lambda)x +  \left(\frac{2-3\lambda+\lambda^3}{4}\right)^{1/3} = (1+\lambda)x +  c_1(\lambda).
\end{equation}
Inserting \eqref{eq:q_lambda} into \eqref{eq:lambda_back} to find
\begin{equation}
	H(s,A,x)\geq (1+\lambda)\alpha x+c_1(\lambda)\alpha +\frac{11}{45}sA-1.
\end{equation}
Suppose there exists \(\eta\in [0,\alpha]\) with
\begin{equation*}
	\frac{7}{135}sA +\eta \leq [k(s;A)]_+.
\end{equation*}
We then can find \(\lambda = \lambda(\eta)\in [-1,1]\) such that
\begin{equation*}
	\alpha c_1(\lambda)= \alpha-\eta,
\end{equation*}
and arrive at
\begin{equation}\label{eq:HwithEta}
	\begin{split}
		H(s,A,x) & \geq (1+\lambda)\alpha x+\alpha-\eta +\frac{11}{45}sA-1   \\
		         & =(1+\lambda)\alpha x -\left(\frac{7}{135}sA +\eta\right).
	\end{split}
\end{equation}
Since \(\eta\in [0,\alpha]\) is admissible we conclude from \eqref{eq:HwithEta} 
\begin{equation}\label{eq:low_bnd_on_LSA}
	L(s;A) \geq \inf_{x>0} \frac{[k(s;A)]_++[(1+\lambda(\eta))\alpha x- [k(s;A)]_+]_+}{x} \geq (1+\lambda)\alpha.
\end{equation}
Recall that in the expression above \(\lambda=\lambda(\eta)\) is function of \(\eta/\alpha\) uniquely determined by
\begin{equation}\label{eq:defines_lambda}
	c_1(\lambda)=\left(\frac{2-3\lambda+\lambda^3}{4}\right)^{1/3}=1-\frac{\eta}{\alpha}.
\end{equation}
We aim to find a lower bound on \(u(\lambda)=1+\lambda \in [0,2]\) in terms of \(z^2=\eta/\alpha \in [0,1]\). Inserting \(u,z\) into \eqref{eq:defines_lambda} we equivalently have
\begin{equation}\label{eq:defining_lambda_uz}
	u^2(3-u)=4z^2(3-3z^2+z^4).
\end{equation}
The mapping \(h(u)=u^2(3-u)\) is increasing in \(u\in [0,2]\) and thus to prove \(u\geq u_0\) for \(u_0\in [0,2]\) it suffices to show \(h(u)\geq h(u_0)\). We choose
\begin{equation*}
	u_0=2z+\frac{2}{5}z^2(1-z^2),
\end{equation*}
then \(u_0\in [0,2]\) and show
\begin{equation}\label{eq:proof_low_bnd}
	4z^2(3-3z^2+z^4)-h(u_0)  \geq 0.
\end{equation}
Expanding \(h(u_0)\) in \eqref{eq:proof_low_bnd} we find the equivalent expression
\begin{equation*}
	-\frac{4}{125}(z-1)^2z^3(2z^7+4z^6-34z^4-47z^3+40z-100)\geq 0
\end{equation*}
which is for \(z\in\{0,1\}\) true and thus for \(z\in(0,1)\) equivalent to
\begin{equation*}
	2z^7+4z^6-34z^4-47z^3+40z-100 \leq  0
\end{equation*}
which is immediate for \(z\in[0,1]\). Consequently \eqref{eq:proof_low_bnd} holds and from \eqref{eq:defining_lambda_uz} we find
\begin{equation}\label{eq:low_bnd_on_lambda}
	1+\lambda = u \geq u_0=2\left(\frac{\eta}{\alpha}\right)^{1/2}+\frac{2}{5}\frac{\eta}{\alpha}\left(1-\frac{\eta}{\alpha}\right).
\end{equation}
Inserting \eqref{eq:low_bnd_on_lambda} into the right hand side of \eqref{eq:low_bnd_on_LSA} we arrive at
\begin{equation}\label{eq:Lbnd}
	L(s;A)\geq 2\sqrt{\eta\alpha}+\frac{2}{5}\eta\left(1-\frac{\eta}{\alpha}\right).
\end{equation}
We fix now \(A=15/2\) and write \(s=r^3\). We provide an admissible \(\eta=\eta(r)\). Recall that \(\alpha=\alpha(r) =1- \frac{20}{9} r^3  \) and we earlier supposed
\begin{equation}\label{eq:conditions}
	\frac{7}{18} r^3 +\eta(r) \leq \left[k\left(r^3;\frac{15}{2}\right)\right]_+, \quad  \eta(r)\in[0,\alpha(r)].
\end{equation}
We choose
\begin{equation}\label{eq:choice_of_eta}
	\eta(r) = \left(1-\frac{17}{7}r\right)\left(1-\frac{2}{7}r\right).
\end{equation}
Clearly \(\eta(r)\geq0\) on this interval, while
\begin{equation*}
	\alpha(r)-\eta(r) = -\frac{1}{441}r\left(980r^2+306r-1197\right)\geq 0, \quad 0\leq r \leq\frac{7}{17}.
\end{equation*}
We now verify the first condition in \eqref{eq:conditions}. From the definition of \(k(r^3;7.5)\) in Lemma~\ref{lem:k_lemma} we have
\begin{equation}\label{eq:k_est}
	\begin{split}
		k(r^3;7.5) -  \frac{7}{18}r^3 & =1+r(1-r^3)^{2/3}-\frac{9}{2}c_0(r^3) + \frac{10}{9}r^3 \\
		                              & \geq 1 +r-r^4-\frac{9}{2}c_0(r^3) + \frac{10}{9}r^3,
	\end{split}
\end{equation}
with
\begin{equation}\label{eq:c_0_est}
	c_0(r^3) = 3^{-2/3}r^{-2}[1-(1-r^3)^{5/3}] \leq \frac{5}{3^{5/3}}r- \frac{5}{3^{8/3}}r^4
\end{equation}
where we have used Taylor's theorem and that the third derivative of
\begin{equation*}
	x\mapsto1-(1-x)^{5/3}
\end{equation*}
is negative. Consequently by inserting \eqref{eq:c_0_est} into \eqref{eq:k_est} we arrive at
\begin{equation*}
	\begin{split}
		k(r^3;7.5) -  \frac{7}{18}r^3 & \geq1 +\left(1- \frac{45}{2\cdot 3^{5/3}}\right) r  + \frac{10}{9}r^3+ \left(\frac{45}{2\cdot 3^{8/3}}-1\right)r^4 \\
		                              & \geq 1 +\left(1- \frac{45}{2\cdot 3^{5/3}}\right) r  + \frac{10}{9}r^3
	\end{split}
\end{equation*}
For our choice of \(\eta\) we have
\begin{equation*}
	\eta(r)=1-\frac{19}{7}r+\frac{34}{49}r^2.
\end{equation*}
Thus, to prove the first condition in \eqref{eq:conditions} it suffices to check
\begin{equation*}
	1+\frac{19}{7}- \frac{45}{2\cdot 3^{5/3}}-\frac{34}{49}r + \frac{10}{9}r^2\geq 0
\end{equation*}
which holds for any \(r\in \R\). This completes the proof of the lower bound in \eqref{eq:Lbnd} at \(A=7.5\) and for \(\eta\) chosen in \eqref{eq:choice_of_eta} for any \begin{equation*}
	0< s=r^3\leq \left(\frac{7}{17}\right)^3
\end{equation*}
We now estimate the integral \(I(7.5)\). Note by using \eqref{eq:Lbnd} together with $\alpha(r)-\eta(r)\geq 0$ and $\alpha(r)\leq 1$ we have
\begin{equation*}
	\begin{split}
		I(7.5) & =\int_{0}^{1/8} L\left(s;\frac{15}{2}\right) ds                                                                       \\
		       & \geq 6\int_0^{\frac{7}{17}} r^2\sqrt{\alpha(r)\eta(r)}dr  + \frac{6}{5}\int_0^{7/17} r^2\eta(r)\frac{\alpha(r)-\eta(r)}{\alpha(r)}dr\\
                &\geq 6\int_0^{\frac{7}{17}} r^2\sqrt{\alpha(r)\eta(r)}dr  + \frac{6}{5}\int_0^{7/17} r^2\eta(r)(\alpha(r)-\eta(r))dr.
	\end{split}
\end{equation*}
For the first integral, we use the inequality
\begin{equation*}
	\sqrt{\alpha(r)\eta(r)} \geq \left(1-\frac{32}{27}r^3\right)\left(1-\frac{16}{105}r\right)\sqrt{1-\frac{17}{7}r}, \quad r\in [0,7/17].
\end{equation*}
The inequality follows from the following two inequalities that follow directly by squaring on the relevant interval:
\begin{equation}
    \sqrt{1- \frac{20}{9}r^3} \geq 1-\frac{32}{27}r^3, \quad \text{ and } \quad \sqrt{1- \frac{2}{7}r} \geq 1-\frac{16}{105}r.
\end{equation}
The remaining integrals after this estimate can be computed explicitly, and we find
\begin{equation*}
	J_1=6\int_0^{\frac{7}{17}} r^2\sqrt{\alpha(r)\eta(r)}dr \geq \frac{8}{135}.
\end{equation*}
Expanding the polynomial integrand and integrating termwise gives
\begin{equation*}
	J_2= \frac{6}{5}\int_0^{7/17} r^2\eta(r)(\alpha(r)-\eta(r))dr \geq \frac{1}{270}
\end{equation*}
and consequently
\begin{equation*}
	I(7.5)\geq \frac{8}{135}+\frac{1}{270}=\frac{1}{16} +\frac{1}{2160}> \frac{1}{16}.
\end{equation*}
\qed 
\section*{Acknowledgments}
I thank Martin Dam Larsen, Vincent Louatron and Léo Morin for helpful discussions. I acknowledge support by the ERC
Advanced Grant MathBEC, grant agreement no.~101095820, and by the
VILLUM Foundation, grant no.~10059.

\bibliographystyle{plain}
\bibliography{ref}

\end{document}